\documentclass[10pt,a4paper]{article}
\usepackage[utf8]{inputenc}
\usepackage{amsmath}
\usepackage{amsfonts}
\usepackage{amssymb}
\usepackage{amsthm}
\usepackage{color}
\usepackage{enumerate}
\usepackage{marvosym}
\usepackage{hyperref}
\usepackage{tikz}
 \newtheorem{thm}{Theorem}[section]

 \newtheorem*{thm*}{Theorem}
 \newtheorem{cor}[thm]{Corollary}
 \newtheorem{lem}[thm]{Lemma}
 \newtheorem{prop}[thm]{Proposition}
 \theoremstyle{definition}
 
 \theoremstyle{remark}
 \newtheorem{rem}[thm]{Remark}
 
 \newtheorem{ex}{Example}
 
 \newtheorem{problem}{Problem}
 \numberwithin{equation}{section}
 
\newcommand{\vertiii}[1]{{\left\vert\kern-0.25ex\left\vert\kern-0.25ex\left\vert #1
  \right\vert\kern-0.25ex\right\vert\kern-0.25ex\right\vert}}

\DeclareMathOperator{\bdo}{BDO}
\DeclareMathOperator{\band}{band}

\title{Rotationally-continuous operators on the Fock space}
\author{Robert Fulsche\footnote{fulsche@math.uni-hannover.de}}
\begin{document}
\maketitle
\begin{abstract}
    We study operators on the Fock space for which conjugation by the rotation operators induces a continuous action of the circle group. We prove that this class of operators can be identified with the space of band-dominated operators on $\ell^2(\mathbb N_0)$ by mapping the operators to their matrix representations with respect to the standard orthonormal basis. Further, we prove that the intersection of this class with the Toeplitz algebra of the Fock space agrees, in the same manner, with the band-dominated operators on $\ell^2(\mathbb N_0)$ such that all the diagonals of the matrix are sequences which are uniformly continuous with respect to the square-root metric.
\end{abstract}

\section{Introduction}

By $\mu_t$ we denote the Gaussian measure $d\mu_z(z) = (\pi t)^{-1} e^{-\frac{|z|^2}{t}}~dz$ on $\mathbb C$, where $t > 0$ is fixed. Then, the Fock space $F_t^2$ consists of all entire functions in $L^2(\mathbb C, \mu_t)$. The investigation of bounded linear operators on this space has now a long tradition and goes back at least to the works of Berger and Coburn \cite{Berger_Coburn1986, Berger_Coburn1987, Berger_Coburn1994}. Since then, the study of operators on the Fock space has been a recurring theme, see, e.g.,  \cite{Bauer_Isralowitz2012, Xia} for some of the most important results. A particular direction of research studies linear operators on the Fock space which are invariant with respect to certain group actions, as well as algebras generated by such operators \cite{Dewage_Olafsson2022, Esmeral_Maximenko2016, Esmeral_Vasilevski2016}. On the other hand, it was observed that another aspect of group actions plays an important role in the theory of Fock space operators: The space of operators which are continuous with respect to a certain action of the underlying phase space are a very well-behaved class and, somewhat surprisingly, this space agrees with the $C^\ast$-algebra generated by all Toeplitz operators with bounded symbols \cite[Theorem 3.1]{Fulsche2020}. More recently, the author together with Raffael Hagger investigated the operators which satisfy a weaker form of continuity \cite{Fulsche_Hagger2025}. More precisely, there they showed (up to unitary equivalence) that the operators on the Fock space which are continuous with respect to the above mentioned group action of the phase space, restricted to a Lagrangian subspace, agree with the class of band-dominated operators on $L^2(\mathbb R)$ in a certain way.

The present paper will continue this line of research. Here, we will consider those operators on which the group action of the circle $\mathbb T$, implemented by adjoining with the operators of rotation $U_\zeta g(z) = g(\zeta z)$, $g \in F_t^2$ and $\zeta \in \mathbb T$, acts strongly continuously. This space is a $C^\ast$-algebra which we denote by $C_R(F_t^2)$. We will show that $C_R(F_t^2)$ naturally agrees with the $C^\ast$-algebra of band-dominated operators on $\ell^2(\mathbb N_0)$ by identifying an operator on $F_t^2$ with their respective matrix representation with respect to the standard orthonormal basis of $F_t^2$. This will be our first main result, Theorem \ref{thm_CR_BDO} in the main text of the paper. Building upon this, as well as on previous work concerning radial operators on the Fock space (which we will later denote as $\operatorname{band}_0$ for reasons that will become apparent later), we will continue to prove that the intersection of the above-mentioned algebra with the $C^\ast$-algebra generated by all Toeplitz operators with bounded symbols, denoted $\mathcal T(F_t^2)$, can be identified with the space of band-dominated operators on $\ell^2(\mathbb N_0)$ which satisfy a uniform continuity condition with respect to the square-root metric on all of their diagonals. The precise formulation of this will be contained in Theorem \ref{thm2}. The relations between all the spaces mentioned above is briefly summarized in the diagram shown in Figure \ref{fig:1}.
\begin{figure}
    \centering
    \begin{tikzpicture}
        \draw[rounded corners] (0, 0) rectangle (12, 6);
        \draw[rounded corners] (0.5, 0.5) rectangle (8, 5);
        \draw[rounded corners] (3, 1) rectangle (11.5, 4.5);
        \draw (6, 2.75) circle (1);
        \node at (6.25, 2.76) {$\mathcal K(F_t^2)$};
        \node at (6, 5.5) {$\mathcal L(F_t^2)$};
        \node at (1.5, 4.5) {$C_R(F_t^2)$};
        \node at (10.5, 4) {$\mathcal T(F_t^2)$};
        \draw (3.6, 2.75) ellipse (2 and 1);
        \node at (3.6, 2.75) {$\operatorname{band}_0$};
    \end{tikzpicture}
    \caption{Diagram showing the inclusion relations among the spaces considered in this paper}
    \label{fig:1}
\end{figure}

The organization of the present paper is straightforward: Section \ref{sec:2} will spell out all the preliminaries necessary for rigorously describing the setting described above. Further, our main results and their proofs will be written out there. Section \ref{sec:3} contains some generalizations of the results obtained in Section \ref{sec:2}, as well as some discussion of open problems.

\section{Rotationally-continuous operators on the Fock space}\label{sec:2}
Let $F_t^2 = F_t^2(\mathbb C)$ be the Fock space of holomorphic functions on $\mathbb C$ which are square-integrable with respect to the Gaussian measure
\begin{equation*}
d\mu_t(z) = \frac{1}{\pi t} e^{-\frac{|z|^2}{t}}~dz.
\end{equation*}
It is well-known that $F_t^2$ is a Hilbert space with inner product 
\begin{align*}
    \langle f, g\rangle_t := \int_{\mathbb C} f(z) \overline{g(z)}~d\mu_t(z).
\end{align*}
A standard reference for the Fock space and properties of certain operators on it is the book by Zhu \cite{Zhu}. This book will also be our main reference whenever we use a fact about the Fock space that we do not prove here. In the following, we will denote by $\mathcal L(\mathcal H)$ the bounded linear operators on the Hilbert space $\mathcal H$ and by $\mathcal K(\mathcal H)$ the class of compact operators on $\mathcal H$.

The circle group $\mathbb T$ (considered as a subset of $\mathbb C$) acts on $F_t^2$ via the unitary operators $U_\zeta$, which are defined as:
\begin{align*}
U_{\zeta} g(z) = g(\zeta z), \quad \zeta \in \mathbb T, ~z \in \mathbb C, ~g \in F_t^2.
\end{align*}
Having these operators at hand, we can already define the main object of study of the present paper, namely
\begin{align*}
C_R(F_t^2) :&= \{ A \in \mathcal L(F_t^2); ~\| U_{\zeta} A U_{\zeta}^\ast - A\|_{op} \to 0, ~\zeta \to 1\},
\end{align*}
the space of bounded operators which are continuous with respect to the rotation action. $C_R(F_t^2)$ is a $C^\ast$-algebra. We refer to the elements of this space as rotationally-continuous operators. Clearly, the radial operators (i.e., operators $A \in \mathcal L(F_t^2)$ for which $\zeta \mapsto U_{\zeta} A U_{\zeta}^\ast$ is constant) are contained in there.

The operators $U_\zeta$ satisfy $U_{\zeta}^\ast = U_{\overline{\zeta}} = U_{\zeta}^{-1}$. Further, $\zeta \mapsto U_\zeta$ is continuous with respect to the strong operator topology. Therefore, for any $f \in L^1(\mathbb T)$ and $A \in \mathcal L(F_t^2)$ we can define the operator
\begin{align*}
f \ast_{\mathbb T} A := \int_{\mathbb T} f(\zeta) U_\zeta A U_\zeta^{-1} d\zeta,
\end{align*}
where the integral is understood in strong operator topology, i.e., for every $g \in F_t^2$ we have
\begin{align*}
f \ast_{\mathbb T} A(g) = \int_{\mathbb T} f(\zeta) U_\zeta A U_\zeta^{-1}(g) d\zeta,
\end{align*}
where the latter expression is now defined as a Bochner integral in $F_t^2$. Here, the measure $d\zeta$ is normalized such that $\int_{\mathbb T} d\zeta = 1$. We will later make use of the fact that this Bochner integral agrees with the corresponding weak integral, or written in more explicit terms:
\begin{align}\label{eq:weak_integral}
    \langle f \ast_{\mathbb T} A g, h \rangle_{F_t^2} = \int_{\mathbb T} f(\zeta) \langle U_\zeta A U_\zeta^{-1} g, h\rangle_{F_t^2}~d\zeta, \quad g, h \in F_t^2.
\end{align}
For details on Bochner and weak integration, we refer to \cite[Section II]{Diestel_Uhl1977}. We want to mention that this ``convolution'' is a special case of the convolutions introduced in \cite{Dewage_Mitkowski2025}, namely by choosing the group $G$ appearing there to be $\mathbb T$.

 For $f \in L^1(\mathbb T)$ and $A \in C_R(F_t^2)$, $f \ast_{\mathbb T} A$ is even defined as a Bochner integral in $C_R(F_t^2)$, hence $f \ast_{\mathbb T} A \in C_R(F_t^2)$ in this case.

For $k \in \mathbb Z$ and $A \in \mathcal L(F_t^2)$ we define the ``Fourier coefficients'' $\hat{A}(k)$ by
\begin{equation*}
\hat{A}(k) := f_k \ast_{\mathbb T} A,
\end{equation*}
where $f_k$ is the function
\begin{equation*}
f_k(\zeta) = \zeta^k, \quad \zeta \in \mathbb T.
\end{equation*}
Recall that for $n \in \mathbb N$ the Fej\'{e}r kernel $F_n(\zeta)$ is defined by
\begin{equation*}
F_n(\zeta) := \frac{1}{n} \sum_{k=0}^{n-1} \sum_{s=-k}^k \zeta^s.
\end{equation*}
As is well-known, the Fej\'{e}r kernel acts as an approximate identity on $C(\mathbb T)$: For any $f \in C(\mathbb T)$ we have, with convergence in uniform norm:
\begin{equation*}
F_n \ast f \to f, \quad n \to \infty.
\end{equation*}
See, for example, \cite[Theorem 5.2]{Stein_Shakarchi} for a proof of this.
As a consequence, it is easily verified that convolution by the Fej\'{e}r kernel also acts as an approximate identity of $L^1(\mathbb T)$. Further important properties of the Fej\'{e}r kernel are that it is positive and normalized:
\begin{equation*}
F_n(\zeta) \geq 0 \text{ for } \zeta \in \mathbb T, \quad \int_{\mathbb T} F_n(\zeta) d\zeta = 1.
\end{equation*}
\begin{lem}\label{lemma_approx_id}
For $A \in C_R(F_t^2)$ we have
\begin{equation*}
F_n \ast_{\mathbb T} A \to A, \quad n \to \infty
\end{equation*}
in operator norm.
\end{lem}
\begin{proof}
Since the Fej\'{e}r kernel is normalized, it is
\begin{equation*}
A = \int_{\mathbb T} F_n(\zeta) A ~d\zeta.
\end{equation*}
Write $h_A(\zeta) := \| U_{\zeta} A U_{\zeta}^{-1} - A\|$. Then, $h_A$ is a continuous function on $\mathbb T$ as soon as $A \in C_R(F_t^2)$. Therefore, we observe that
\begin{align*}
\|F_n \ast_{\mathbb T} A - A\| &= \left \| \int_{\mathbb T} F_n(\zeta) (U_{\zeta} A U_{\zeta}^{-1} - A)  d\zeta \right \|\\
&\leq \int_{\mathbb T} F_n(\zeta) \| U_{\zeta} A U_{\zeta}^{-1} - A\| d\zeta\\
&= \int_{\mathbb T} F_n(\zeta) h_A(\zeta)~d\zeta\\
&= F_n \ast h_A(1) \to h_A(1) = 0
\end{align*}
as $n \to \infty$.
\end{proof}
Writing out $F_n \ast_{\mathbb T} A$, we obtain
\begin{align*}
F_n \ast_{\mathbb T} A &= \frac{1}{n} \sum_{k=0}^{n-1} \sum_{s=-k}^k f_s \ast_{\mathbb T} A\\
&= \frac{1}{n} \sum_{k=0}^{n-1} \sum_{s = -k}^k \hat{A}(s).
\end{align*}
Hence, the previous lemma states that every $A \in C_R(F_t^2)$ is the limit of a weighted sum of its Fourier coefficients. Let us further investigate those coefficients. For this, let us recall the standard orthonormal basis of $F_t^2$: For every $n \in \mathbb N_0$ we the element $e_n^t$ of the standard orthonormal basis is given by the polynomial
\begin{equation*}
e_n^t(z) = \frac{1}{\sqrt{n!t^n}}z^n.
\end{equation*}
This orthonormal basis can now be used as follows: We can identify every element of $\mathcal L(F_t^2)$ with an operator in $\mathcal L(\ell^2(\mathbb N_0))$ in the natural way: $A \in \mathcal L(F_t^2)$ corresponds to the operator in $\mathcal L(\ell^2(\mathbb N_0))$ given by an infinite matrix, which we will denote (for clearly distinguishing between an operator on $F_t^2$ and its matrix) as $\mathfrak M(A)$. More precisely:
\begin{equation*}
\mathfrak M(A) = (\langle Ae_m^t, e_n^t\rangle_t )_{m, n = 0}^\infty \in \mathcal L(\ell^2(\mathbb N_0)).
\end{equation*}
In the following, we will always denote operators on $\ell^2(\mathbb N_0)$ as $\mathfrak M(A), \mathfrak M(B), \mathfrak M(C)$ and so forth and write $A, B, C,...$ for corresponding operators on $F_t^2$ given by the respective matrices with respect to the standard basis. For example, the unitary operators $U_\zeta$ act on the standard basis as $U_\zeta e_m^t = \zeta^m e_m^t$ for every $\zeta \in \mathbb T$ and $m \in \mathbb N_0$. In particular, these operators correspond to diagonal matrices, i.e., to multiplication operators on $\ell^2(\mathbb N_0)$:
\begin{align}\label{eq:matrix_rotation}
    \mathfrak M(U_\zeta) = \operatorname{diag}(1, \zeta, \zeta^2, \zeta^3, \dots), \quad \zeta \in\mathbb T.
\end{align}
Here, the notion $\operatorname{diag}(a_0, a_1, a_2, \dots)$ denotes the infinite matrix having only the values $a_j$ on the diagonal and zero everywhere else.

For $k \in \mathbb Z$ we define
\begin{equation*}
\band_k := \{ \mathfrak M(C) \in \mathcal L(\ell^2(\mathbb N_0)); ~ \langle Ce_m^t, e_n^t\rangle_t = 0 \text{ unless } m-n = k\},
\end{equation*}
i.e., $\band_k$ consists of those bounded linear operators on $\ell^2(\mathbb N_0)$ whose infinite matrix is only supported on the $k$th off-diagonal.
\begin{lem}\label{lemma2}
\begin{enumerate}[1)]
\item Let $A \in \mathcal L(F_t^2)$. Then, $\mathfrak M(\hat{A}(k)) \in \band_k$. 
\item Let $A \in \mathcal L(F_t^2)$ such that $\mathfrak M(A) \in \band_k$ for some $k \in \mathbb Z$. Then, $A \in C_R(F_t^2)$ and $\hat{A}(l) = \delta_{k,l}A$ for any $l \in \mathbb Z$. Here, $\delta_{k,l}$ is the Kronecker delta, i.e.~it equals $1$ if $k = l$ and is zero otherwise.
\item If $\mathfrak M(A) \in \band_{k_1}$ and $\mathfrak M(B) \in \band_{k_2}$, then $\mathfrak M(A)\mathfrak M(B) = \mathfrak M(AB) \in \band_{k_1 + k_2}$ and $\mathfrak M(A)^\ast = \mathfrak M(A^\ast) \in \band_{-k_1}$.
\end{enumerate}
\end{lem}
\begin{proof}
\begin{enumerate}[1)]
\item We clearly have the identity $U_{\zeta} e_m^t = \zeta^m e_m^t$. Therefore, utilizing equation \eqref{eq:weak_integral}:
\begin{align*}
\langle \hat{A}(k) e_m^t, e_n^t\rangle_t &= \int_{\mathbb T} \zeta^k \langle AU_{\zeta}^{-1} e_m^t, U_{\zeta}^{-1} e_n^t\rangle_t ~d\zeta\\
&= \int_{\mathbb T} \zeta^k \langle \zeta^{-m} Ae_m^t, \zeta^{-n} e_n^t\rangle_t ~d\zeta\\
&= \int_{\mathbb T} \zeta^{k-m+n} \langle Ae_m^t, e_n^t\rangle_t ~d\zeta\\
&= \begin{cases}
\langle Ae_m^t, e_n^t\rangle_t, \quad & m-n = k,\\
0, \quad & \text{otherwise}.
\end{cases}
\end{align*}
\item Assume $k \leq 0$ and hence $m - k \geq 0$ for any $m \geq 0$. Otherwise, the same arguments as below yield the result, only with the roles of $m$ and $n$ exchanged.

The assumption says that $A$ is such that $\langle A e_m^t, e_n^t\rangle_t = 0$ unless $m-n = k$. For arbitrary $f, g \in F_t^2$ we can write
\begin{align*}
f = \sum_{m=0}^\infty f_m e_m^t, \quad g = \sum_{n = 0}^\infty g_n e_n^t.
\end{align*}
Thus, we have
\begin{equation*}
\langle Af, g\rangle_t = \sum_{m=0}^\infty f_m \overline{g_{m-k}} \langle Ae_m^t, e_{m-k}^t\rangle_t.
\end{equation*}
Further, since
\begin{equation*}
U_{\zeta} f = \sum_{m = 0}^\infty f_m \zeta^m e_m^t
\end{equation*}
and similarly for $g$, we get
\begin{align*}
\langle U_{\zeta} A U_{\zeta}^{-1} f, g\rangle_t &= \langle A U_{\zeta}^{-1} f, U_{\zeta}^{-1} g\rangle_t\\
&= \sum_{m=0}^\infty f_m \zeta^{-m} \overline{g_{m-k}}\zeta^{m-k} \langle Ae_m^t, e_{m-k}^t\rangle_t\\
&= \zeta^{-k} \sum_{m=0}^\infty f_m \overline{g_{m-k}}\langle Ae_m^t, e_{m-k}^t\rangle_t\\
&= \zeta^{-k} \langle Af, g\rangle_t, \tag{1}
\end{align*}
hence
\begin{align*}
\langle (A- U_{\zeta} A U_{\zeta}^{-1})f, g\rangle_t = (1 - \zeta^{-k})\langle Af, g\rangle_t.
\end{align*}
From this one easily sees that
\begin{align*}
\| A - U_{\zeta} A U_{\zeta}^{-1}\| &= \sup_{\| f\|= \| g\| = 1} |\langle (A - U_{\zeta} A U_{\zeta}^{-1})f, g\rangle_t|\\
&= |1 - \zeta^{-k}| \sup_{\| f\| = \| g\| = 1} |\langle Af, g\rangle_t|\\
&\leq |1-\zeta^{-k}| \| A\| \to 0, \quad \zeta \to 1.
\end{align*}
This proves $A \in C_R(F_t^2)$. It remains to show that $\hat{A}(l) = \delta_{k,l}A$. Using Equation (1) we get
\begin{align*}
\hat{A}(l) &= f_l \ast_{\mathbb T} A = \int_{\mathbb T} \zeta^l U_{\zeta} A U_{\zeta}^{-1}~d\zeta\\
&= \int_{\mathbb T} \zeta^{l-k} A ~d\zeta\\
&= \begin{cases}
A, \quad &l = k,\\
0, \quad &\text{otherwise}.
\end{cases}
\end{align*}
\item Follows immediately from the definition of the matrix product.\qedhere
\end{enumerate}
\end{proof}
Recall that the algebraic sum $\oplus_{k \in \mathbb Z} \band_k \subset \mathcal L(\ell^2(\mathbb N_0))$ is usually called the algebra of \emph{band operators}. Their norm closure is denoted $\bdo(\ell^2(\mathbb N_0))$, the \emph{band-dominated operators}. We now obtain:
\begin{thm}\label{thm_CR_BDO}
Let $A \in \mathcal L(F_t^2)$. Then, we have $A \in C_R(F_t^2)$ if and only if $\mathfrak M(A) \in \bdo(\ell^2(\mathbb N_0))$.
\end{thm}
\begin{proof}
Let $A \in C_R(F_t^2)$. Then, by Lemma \ref{lemma_approx_id}, $\mathfrak M(A)$ can be approximated by band operators, hence $\mathfrak M(A) \in \bdo(\ell^2(\mathbb N_0))$. On the other hand, assume $A \in \mathcal L(F_t^2)$ is such that $\mathfrak M(A) \in \bdo(\ell^2(\mathbb N_0))$. Let $\varepsilon > 0$. Then, there is some band operator $\mathfrak M(B)$ such that $\| \mathfrak M(A) - \mathfrak M(B)\| < \varepsilon$. By definition, there are operators $\mathfrak M(B_k) \in \band_k$ and $N \in \mathbb N$ such that
\begin{equation*}
\mathfrak M(B) = \sum_{k = -N}^N \mathfrak M(B_k).
\end{equation*}
By Lemma \ref{lemma2}, each of the $B_k$ is in $C_R(F_t^2)$, hence $B$ is. Since $\varepsilon > 0$ was arbitrary, $A$ is also in $C_R(F_t^2)$.
\end{proof}
\begin{rem}
\begin{enumerate}
    \item $\bdo(\ell^2(\mathbb N_0))$ is a well-studied algebra. Its operators are well understood, e.g., compactness and Fredholm properties can be characterized in terms of limit operators. For a detailed account on the theory of band-dominated operators we refer to \cite{Lindner2006, Rabinovich_Roch_Silbermann2004}. We remark that usually one considers $\bdo(\ell^2(\mathbb Z))$, i.e., two-sided infinite matrices which are band-dominated, instead of $\bdo(\ell^2(\mathbb N_0))$. Nevertheless, both spaces are closely related: Elements $A$ of $\bdo(\ell^2(\mathbb N_0))$ are of the form $PBP$, where $B \in \bdo(\ell^2(\mathbb N_0))$ and $P$ is the orthogonal projection from $\ell^2(\mathbb Z)$ to $\ell^2(\mathbb N_0)$ (where the latter is understood in the natural way as a subspace of $\ell^2(\mathbb N_0)$). 
    \item In \cite[Theorem 2.1.6]{Rabinovich_Roch_Silbermann2004}, another interesting characterization of $\operatorname{BDO}(\ell^2(\mathbb Z))$ is derived, the roots of which go back to the paper \cite{Kozak_Simonenko1980}: It is proven that
    \begin{align*}
        \operatorname{BDO}(\ell^2(\mathbb Z)) = \{ T \in \mathcal L(\ell^2(\mathbb Z)): ~\lim_{r \to \infty} ~\sup_{X, Y \subset \mathbb Z: ~d(X,Y) \geq r} \| M_{\mathbf 1_X} T M_{\mathbf 1_Y}\| = 0\},
    \end{align*}
    where $d(X, Y)$ is the distance of the set $X$ and $Y$. Writing $\operatorname{BDO}(\ell^2(\mathbb N_0)) = P \operatorname{BDO}(\ell^2(\mathbb Z)) P$, where $P$ is as above the orthogonal projection from $\ell^2(\mathbb Z)$ to $\ell^2(\mathbb N_0)$, it is not hard to derive an analogous characterization for $\operatorname{BDO}(\ell^2(\mathbb N_0))$.
    \item While the above result gives an operator-theoretic description on the Fock space of $\operatorname{BDO}(\ell^2(\mathbb N_0))$, one could also try to find a Fock space description of the dense subalgebra of band operators. We could not come up with such a description and hence pose this as an open problem: Find an operator-theoretic description with methods of the Fock space of
    \begin{align*}
        \mathfrak M^{-1}(\bigoplus_{k \in \mathbb Z} \operatorname{band}_k).
    \end{align*}
\end{enumerate}
\end{rem}

As a consequence of the theorem, we have the following result:
\begin{thm}[{\cite[Proposition 2.1.7]{Rabinovich_Roch_Silbermann2004}}]
$\bdo(\ell^2(\mathbb N_0))$ contains $\mathcal K(\ell^2(\mathbb N_0))$. 
\end{thm}
In particular, $C_R(F_t^2)$ of course contains $\mathcal K(F_t^2)$. We want to note that this can also be verified directly in a straightforward manner, using that the map $\zeta \mapsto U_\zeta$ is continuous in strong operator topology, hence each rank one operator is contained in $C_R(F_t^2)$. Since this has probably not been observed in the literature, we explicitly mention the result:
\begin{cor}
    $C_R(F_t^2)$ contains $\mathcal K(F_t^2)$. 
\end{cor}

We now want to spend some lines on relating Theorem \ref{thm_CR_BDO} to existing results in the literature. As a starting point for this, we note the following:
\begin{lem}
    For each $\zeta \in \mathbb T$ it is $\mathfrak M(U_\zeta) = \operatorname{diag}(1, \zeta, \zeta^2, \zeta^3, \dots)$, where the latter denotes the diagonal matrix with the respective values on the diagonal.
\end{lem}
A diagonal matrix of course acts on $\ell^2(\mathbb N_0)$ simply as a multiplication operator. Hence, $U_\zeta$ can be identified with the multiplication operator $M_{\varphi_\zeta}$ on $\ell^2(\mathbb N_0)$, where $\varphi_\zeta(n) = \zeta^n$ for every $n \in \mathbb N_0$. Hence, we see that:
\begin{lem}
    $\mathfrak M(C_R(F_t^2)) = \{ T \in \mathcal L(\ell^2(\mathbb N_0)): ~\| M_{\varphi_\zeta} T - TM_{\varphi_\zeta}\| \to 0 \text{ as } \zeta \to 1\}.$
\end{lem}
Since we have already seen that, using the map $\mathfrak M$, $C_R(F_T^2)$ can be identified with $\bdo(\ell^2(\mathbb N_0))$, we obtain:
\begin{cor}\label{cor_bdo_N}
    The following equality holds true:
    \begin{align*}
        \bdo(\ell^2(\mathbb N_0)) = \{ T \in \mathcal L(\ell^2(\mathbb N_0)): ~\| M_{\varphi_\zeta} T - TM_{\varphi_\zeta}\| \to 0, ~\zeta \to 1\}.
    \end{align*}
\end{cor}
An analogous characterization of $\bdo(\ell^2(\mathbb Z))$ is well-known in the literature, see \cite[Theorem 2.1.6, condition (e)]{Rabinovich_Roch_Silbermann2004} or \cite[Theorem 1.42, condition (v)]{Lindner2006}. Both of these results can easily be reformulated as follows, where $\psi_\zeta(n) = \zeta^n$, $n \in \mathbb Z$ denotes a function on $\mathbb Z$ and $M_{\psi_\zeta}$ the appropriate multiplication operator on $\ell^2(\mathbb Z)$:
\begin{align}\label{equality_bdo_Z}
    \bdo(\ell^2(\mathbb Z)) = \{ T \in \mathcal L(\ell^2(\mathbb Z)): ~\| M_{\psi_\zeta} T - T M_{\psi_\zeta}\| \to 0 \text{ as } \zeta \to 1\}.
\end{align}
We denote again by $P$ the orthogonal projection from $\ell^2(\mathbb Z)$ to $\ell^2(\mathbb N_0)$. Combining equality \eqref{equality_bdo_Z} with the fact that $M_{\varphi_\zeta} P = PM_{\psi_\zeta}$ and further that $\bdo(\ell^2(\mathbb N_0)) = P \bdo(\ell^2(\mathbb Z)) P$, it is not hard to derive Corollary \ref{cor_bdo_N} and hence also Theorem \ref{thm_CR_BDO}. Hence, another proof of Theorem \ref{thm_CR_BDO} can be given by means of Equation \eqref{equality_bdo_Z}. Nevertheless, we want to point out that both approaches are not entirely unrelated: The proof of the inclusion ``$\supseteq$'' in \eqref{equality_bdo_Z} given in \cite{Rabinovich_Roch_Silbermann2004} uses a similar trick utilizing continuity of the group action and properties of the Fej\'{e}r kernel that we used.

Another operator algebra of interest, which also contains $\mathcal K(F_t^2)$, is the Toeplitz algebra $\mathcal T(F_t^2)$. We now remind the reader of its definition:

We let $P_t$ denote the orthogonal projection from $L^2(\mathbb C, \mu_t)$ to $F_t^2$. Then, for any $h \in L^\infty(\mathbb C)$, the \emph{Toeplitz operator} $T_h^t$ is defined on $F_t^2$ by $T_h^t(g) = P_t(hg)$. Obviously, $\| T_h^t\|_{op} \leq \| h\|_\infty$. We want to emphasize that the map $L^\infty(\mathbb C) \ni h \mapsto T_h^t$ is injective \cite[Theorem 4]{Berger_Coburn1986}.

By $\mathcal T(F_t^2)$ we now denote the $C^\ast$-subalgebra of $\mathcal L(F_t^2)$ generated by all Toeplitz operators with bounded symbols:
\begin{align*}
    \mathcal T(F_t^2) := C^\ast(\{ T_h^t \in F_t^2: ~h \in L^\infty(\mathbb C)\}).
\end{align*}
We want to note that, just as $C_R(F_t^2)$, $\mathcal T(F_t^2)$ can be described as the algebra of operators which are continuous with respect to a certain group action, cf.~\cite[Theorem 3.1]{Fulsche2020}. Nevertheless, this will not be important throughout this paper.

It is well-known that $\mathcal K(F_t^2) \subset \mathcal T(F_t^2)$. Since both $\mathcal T(F_t^2)$ and $C_R(F_t^2)$ are $C^\ast$-algebras, we obtain the following:
\begin{prop}\label{Prop:containment_K}
$\mathcal T(F_t^2) \cap C_R(F_t^2)$ is a $C^\ast$-subalgebra of $\mathcal L(F_t^2)$ containing $\mathcal K(F_t^2)$.
\end{prop}
For $h \in L^1(\mathbb C)$ we can also define $R_\zeta h(z) = h(\zeta z)$. Then, $\zeta \mapsto R_\zeta$ acts strongly continuously on $L^1(\mathbb C)$, i.e., for every $h \in L^1(\mathbb C)$ the map $\zeta \mapsto R_\zeta h$ is continuous from $\mathbb T$ to $L^1(\mathbb C)$. This can be verified directly for $h \in C_c(\mathbb C)$ using the dominated convergence theorem and then, for general $h \in L^1(\mathbb C)$, using density of $C_c(\mathbb C)$.

Then, by duality, $\zeta \mapsto R_\zeta$ (defined by the same formula) acts weak$^\ast$ continuous on $L^\infty(\mathbb C)$, meaning that for every $h \in L^\infty(\mathbb C)$ and $g \in L^1(\mathbb C)$ the map $\zeta \mapsto \int_\mathbb C R_\zeta h(w)~g(w)~dw$ is continuous. Hence, for any $h \in L^\infty(\mathbb C)$ and $f \in L^1(\mathbb T)$ we can define the following as an integral in the weak$^\ast$ sense in $L^\infty(\mathbb C)$:
\begin{align*}
f \ast_{\mathbb T} h := \int_{\mathbb T} f(\zeta) R_\zeta(h)~d\zeta \in L^\infty(\mathbb C).
\end{align*}
In analogy to the case of operators, we now set for $h \in L^\infty(\mathbb C)$:
\begin{align*}
\hat{h}(k) := f_k \ast_{\mathbb T} h \in L^\infty(\mathbb C).
\end{align*}
We just want to mention that, for sufficiently nice functions (say, $h$ continuous) one can use polar coordinates to calculate $\widehat{h}(k)$ as:
\begin{align*}
    \widehat{h}(k)(re^{i\theta}) = e^{-ik\theta} \int_{\mathbb T} \zeta^k h(r\zeta)~d\zeta.
\end{align*}

In what follows, we will try to understand the algebra $\mathcal T(F_t^2) \cap C_R(F_t^2)$. An important tool for the study of operators on $F_t^2$, in particular of Toeplitz operators, is the \emph{Berezin transform}: For $A \in \mathcal L(F_t^2)$ we set
\begin{equation*}
\mathcal B(A)(z) = \langle A k_z^t, k_z^t\rangle_t, \quad z \in \mathbb C.
\end{equation*}
Here,
\begin{equation*}
k_z^t(w) = e^{\frac{w \cdot \overline{z}}{t} - \frac{|z|^2}{t}}
\end{equation*}
is the \emph{normalized reproducing kernel}. Note that the Berezin transform is injective \cite[Theorem 2.2]{Stroethoff1997}. For $f \in L^\infty(\mathbb C)$ and $t > 0$ we define
\begin{equation*}
\mathcal B_t(f)(z) := \langle T_f^t k_z^t, k_z^t\rangle_t.
\end{equation*}
\begin{lem} Let $k \in \mathbb Z$. 
\begin{enumerate}[1)]
\item For $f \in L^\infty(\mathbb C)$ we have
\begin{equation*}
\mathcal B_t(\hat{f}(k)) := \widehat{\mathcal B_t(f)}(k).
\end{equation*}
\item For $A \in \mathcal L(F_t^2)$ we have
\begin{equation*}
\mathcal B(\hat{A}(k)) = \widehat{\mathcal B(A)}(k).
\end{equation*}
\item Let $A \in \mathcal L(F_t^2)$. Then, $\mathfrak M(A) \in \band_k$ if and only if the Berezin transform satisfies
\begin{align*}
\mathcal B(A) = \widehat{\mathcal B(A)}(k).
\end{align*}
\end{enumerate}
\end{lem}
\begin{proof}
\begin{enumerate}[1)]
\item This is an application of Fubini's Theorem and a simple integral transform:
\begin{align*}
\mathcal B_t(\hat{f}(k))(z) &= \frac{1}{\pi t} \int_{\mathbb C} \int_{\mathbb T} \zeta^k f(\zeta w) e^{-\frac{|w-z|^2}{t}}~d\zeta ~dw \\
&= \frac{1}{\pi t} \int_{\mathbb T} \zeta^k \int_{\mathbb C} f(\zeta w) e^{-\frac{|w-z|^2}{t}}~dw ~d\zeta\\
&= \frac{1}{\pi t} \int_{\mathbb T} \zeta^k \int_{\mathbb C} f(v) e^{-\frac{|\overline{\zeta}v - z|^2}{t}} ~dv~d\zeta\\
&= \frac{1}{\pi t} \int_{\mathbb T} \zeta^k \int_{\mathbb C} f(v) e^{-\frac{|v - \zeta z|^2}{t}} ~dv~d\zeta\\
&= \widehat{\mathcal B_t(f)}(k)(z).
\end{align*}
\item Since the defining integrals exist in strong operator topology, we have
\begin{align*}
\mathcal B (\hat{A}(k))(z) &= \langle \hat{A}(k) k_z^t, k_z^t \rangle_t\\
&= \int_{S^1} \zeta^k \langle A U_\zeta^{-1}k_z^t, U_\zeta^{-1} k_z^t\rangle_t ~d\zeta\\
&= \int_{S^1} \zeta^k \langle A k_{\zeta z}^t, k_{\zeta z}^t\rangle_t~d\zeta\\
&= \widehat{\mathcal B(A)}(k)(z).
\end{align*}
Here, we used the identity
\begin{align*}
U_{\zeta} k_z^t(w) = e^{\frac{(\zeta w) \cdot \overline{z}}{t} - \frac{|z|^2}{2t}} = k_{\overline{\zeta}z}^t(w).
\end{align*}
\item Assume $A$ is such that $\mathcal B(A) = \widehat{\mathcal B(A)}(k)$. By 2) we have
\begin{equation*}
\mathcal B(A) = \mathcal B(\widehat{A}(k)).
\end{equation*}
Since the Berezin transform is injective, this yields $A = \widehat{A}(k)$, i.e., $\mathfrak M(A) \in \band_k$. The other implication is obvious.\qedhere
\end{enumerate}
\end{proof}
\begin{rem}
Part 3) of the previous lemma generalized the well-known result that an operator is radial if and only if its Berezin transform is radial, cf.~\cite[Theorem 2.5]{Esmeral_Maximenko2016}.
\end{rem}
We define
\begin{align*}
C_R(\mathbb C):= \{ h \in L^\infty(\mathbb C); ~ \zeta \mapsto R_\zeta(h) \text{ is } \| \cdot \|_\infty-\text{continuous}\}.
\end{align*}
For every $h \in C_R(\mathbb C)$, $f_k \ast_{\mathbb T} h = \hat{h}(k)$ exists as a Bochner integral in $C_R(\mathbb C)$ and therefore $\hat{h}(k) \in C_R(\mathbb C)$ in that case. Further, similarly to the case of operators from $C_R(F_t^2)$, $F_n \ast h \to h$ in $L^\infty(\mathbb C)$-norm for such $h$.
\begin{lem}
\begin{enumerate}[1)]
\item Let $h \in C_R(\mathbb C)$. Then, $T_h^t \in C_R(F_t^2)$.
\item Let $A \in C_R(F_t ^2)$. Then, $\mathcal B(A) \in C_R(\mathbb C)$.
\end{enumerate}
\end{lem}
\begin{proof}
\begin{enumerate}[1)]
\item Follows from the following standard estimates:
\begin{align*}
\| T_h^t - U_\zeta T_h^t U_\zeta^{-1}\| &= \| T_h^t - T_{h(\zeta \cdot)}^t\|\\
&\leq \| h - h(\zeta \cdot)\|_\infty\\
&\to 0, \quad \zeta \to 1.
\end{align*}
\item We have
\begin{align*}
\| \mathcal B(A) - \eta_\zeta \mathcal B(A)\| &= \sup_{z \in \mathbb C} |\langle (A - U_\zeta A U_\zeta^{-1})k_z^t, k_z^t\rangle_t\\
&\leq \| A - U_\zeta A U_\zeta^{-1}\|\\
&\to 0, \quad \zeta \to 1,
\end{align*}
hence $\mathcal B(A) \in C_R(\mathbb C)$.\qedhere
\end{enumerate}
\end{proof}
\begin{ex}\label{exampleA}
Let $k > 0$. Then, the function $z \mapsto h_k(z) := \frac{z^k}{|z|^k}$ is defined almost everywhere on $\mathbb C$ and clearly contained in $C_R(\mathbb C)$. The Toeplitz operator $T_{h_k}^t$ is therefore an element of $C_R(F_t^2)$. It is even true that $\mathfrak M(T_{h_k}^t) \in \band_k$, as the following computations show:
\begin{align*}
T_{h_k}^t e_m^t(w) &= \frac{1}{\pi t} \int_{\mathbb C} \frac{1}{\sqrt{t^m m!}} \frac{z^{m+k}}{|z|^k} e^{\frac{w\cdot \overline{z}}{t} - \frac{|z|^2}{t}}~dz\\
&= \frac{1}{\pi \sqrt{t^{m+2} m!}} \int_0^\infty r^{m+1} e^{-\frac{r^2}{t}} \int_0^{2\pi} e^{i \theta(m + k) + \frac{w r}{t}e^{-i\theta}} d\theta ~dr.
\end{align*}
Using the path $\gamma(\theta) = e^{i\theta}$, $\theta \in [0, 2\pi]$ we have
\begin{align*}
\int_0^{2\pi} e^{i\theta(m+k) + \frac{wr}{t}e^{-i\theta}} d\theta = \frac{1}{i}\int_\gamma z^{m+k-1} e^{\frac{wr}{t} \cdot \frac{1}{t}} dz
\end{align*}
and therefore by the Residue Theorem
\begin{align*}
\int_0^{2\pi} e^{i\theta(m+k) + \frac{wr}{t}e^{-i\theta}} d\theta = 2\pi \operatorname{Res}_{z=0}(z^{m+k-1}e^{\frac{wr}{t} \cdot \frac{1}{z}}).
\end{align*}
For the function $z \mapsto z^{m+k-1}e^{\frac{wr}{t} \cdot \frac{1}{z}}$ we of course have the following Laurent series expansion around $0$:
\begin{align*}
z^{m+k-1}e^{\frac{wr}{t} \cdot \frac{1}{z}} &= z^{m+k-1} \sum_{l=0}^\infty \frac{(wr)^l}{t^l l!}\frac{1}{z^l}\\
&= \sum_{l=0}^\infty \frac{(wr)^l}{t^l l!} \frac{1}{z^{l+1-m-k}}.
\end{align*}
Therefore, we obtain for the residue:
\begin{align*}
\operatorname{Res}_{z=0}(z^{m+k-1}e^{\frac{wr}{t} \cdot \frac{1}{z}}) = \frac{(wr)^{m+k}}{t^{m+k}(m+k)!}.
\end{align*}
This now gives
\begin{align*}
T_{h_k}^t e_m^t(w) &= \frac{2w^{m+k}}{\sqrt{t^{m+2} m!} t^{m+k}(m+k)!} \int_0^\infty r^{2m + k + 1}e^{-\frac{r^2}{t}} dr\\
&= \frac{\Gamma(m + \frac{k}{2} + 1)}{(m+k)! \sqrt{m!} \sqrt{t^{m+k}}} w^{m+k}\\
&= e_{m+k}^t(w) \frac{\Gamma(m + \frac{k}{2} + 1)}{\sqrt{m!(m+k)!}}\\
&= c_m^k e_{m+k}^t(w)
\end{align*}
with
\begin{align*}
c_m^k := \frac{\Gamma(m + \frac{k}{2} + 1)}{\sqrt{m!(m+k)!}}.
\end{align*}
By taking adjoints, we obtain now
\begin{align*}
T_{\overline{h_k}}^t e_m^t = \begin{cases}
c_{m-k}^k e_{m-k}^t, \quad &m \geq k,\\
0, \quad &m < k.
\end{cases}
\end{align*}
In particular, $\mathfrak M(T_{h_k}^t) \in \band_k$ and $\mathfrak M(T_{\overline{h_k}}) \in \band_{-k}$ for any $k \geq 1$. By definition, we also let $h_0 = 1$ such that $T_{h_0}^t = I$ and $\mathfrak M(T_{h_0}^t) \in \band_0$.

We want to end this example by mentioning that the operators $T_{h_k}^t$ and $T_{\overline{h_k}}^t$ are within the class of Toeplitz operators with quasi-homogeneous symbol. While the literature regarding Toeplitz operators on the Fock space with such symbols seems sparse, the investigations of such operators on the Bergman space has a long tradition, going back at least to the paper \cite{Cuckovic_Rao1998}. 
\end{ex}
It is our next goal to characterize the membership of $A \in \mathcal L(F_t^2)$ in $\mathcal T(F_t^2)\cap C_R(F_t^2)$ entirely in terms of its matrix coefficients, i.e., in terms of properties of $\mathfrak M(A)$. For doing so, we recall the main result of \cite{Esmeral_Maximenko2016} which we reformulate in our notation:
\begin{thm}[{\cite[Theorem 1.1]{Esmeral_Maximenko2016}}]
Let $A \in \mathcal L(F_1^2)$ be such that $\mathfrak M(A) \in \band_0$. Then, we have $A \in \mathcal T(F_1^2)$ if and only if the sequence $(a_j)_{j=0}^\infty := \langle Ae_j^1, e_j^1\rangle_1$ is uniformly continuous with respect to the metric
\begin{align*}
\rho(m, n) := |\sqrt{m} - \sqrt{n}|
\end{align*}
on $\mathbb N_0$.
\end{thm}
The previous result easily gives rise to the analogous statement for arbitrary $t > 0$:
\begin{cor}\label{Cor_band0}
Let $A \in \mathcal L(F_t^2)$ such that $\mathfrak M(A) \in \band_0$. Then, we have $A \in \mathcal T(F_t^2)$ if and only if the sequence $(a_j)_{j=0}^\infty := (\langle Ae_j^t, e_j^t\rangle_t)_{j=0}^\infty$ is uniformly continuous with respect to the metric $\rho$.
\end{cor}
\begin{proof}
One easily verifies that the operator $V_{\sqrt{t}}$, acting as $V_{\sqrt{t}} g(z) = g(\sqrt{t}z)$, maps
\begin{align*}
V_{\sqrt{t}}: F_t^2 \to F_1^2
\end{align*}
isometrically and satisfies both
\begin{align*}
V_{\sqrt{t}} e_k^{t} = e_k^1
\end{align*}
and
\begin{align}\label{Isometry_V}
V_{1/\sqrt{t}} T_h^1 V_{\sqrt{t}} = T_{V_{1/\sqrt{t}} h}^t.
\end{align}
Combining these facts, one easily sees that $A \in \mathcal L(F_1^2)$ with $\mathfrak M(A) \in \band_0$ if and only if $V_{1/\sqrt{t}} A V_{\sqrt{t}} \in \mathcal L(F_t^2)$ with $\mathfrak M(V_{1/\sqrt{t}} A V_{\sqrt{t}}) \in \band_0$ and the sequences
\begin{align*}
(\langle Ae_j^1, e_j^1\rangle_1)_{j=0}^\infty
\end{align*}
and
\begin{align*}
(\langle V_{1/\sqrt{t}} A V_{\sqrt{t}}e_j^t, e_j^t\rangle_t)_{j=0}^\infty
\end{align*}
are indeed identical. Since the generators of $\mathfrak M(\mathcal T(F_1^2)) \cap \band_0$ and $\mathfrak M(\mathcal T(F_t^2)) \cap \band_0$ are in isometric correspondence by Equation \eqref{Isometry_V}, we have $A \in \mathcal T(F_1^2)$ if and only if $V_{1/\sqrt{t}} A V_{\sqrt{t}} \in \mathcal T(F_t^2)$, which finishes the proof.
\end{proof}

\begin{prop}
Let $A \in \mathcal L(F_t^2)$ such that $\mathfrak M(A) \in \band_k$ for some $k \in \mathbb Z$. If $k \geq 1$, then we have $A \in \mathcal T(F_t^2)$ if and only if the sequence $(\langle Ae_{j+k}^t, e_j^t\rangle_t)_{j=0}^\infty$ is uniformly continuous with respect to $\rho$. If $k \leq -1$, then we have $A \in \mathcal T(F_t^2)$ if and only if the sequence $(\langle Ae_j^t, e_{j-k}^t\rangle_t)_{j=0}^\infty$ is uniformly continuous with respect to $\rho$.
\end{prop}
\begin{proof}
Let us assume $k \geq 1$ and $\mathfrak M(A) \in \band_k$. The other case $k \leq -1$ can be proven identically by simply replacing in the following arguments the operator $T_{h_k}^t$ by $T_{\overline{h_{-k}}}^t$.

We first assume that $(\langle Ae_{j+k}^t, e_j^t\rangle_t)_{j=0}^\infty$ is uniformly continuous with respect to $\rho$. Let
\begin{align*}
b_j := \langle Ae_{j+k}^t, e_j^t\rangle_t
\end{align*}
and define on $F_t^2$ the operator acting diagonally on the standard orthonormal basis as
\begin{align*}
Be_j^t = b_j e_j^t, \quad j \in \mathbb N_0.
\end{align*}
Clearly, $\mathfrak M(B) \in \band_0$. By assumption, the sequence $(b_j)$ is uniformly continuous with respect to $\rho$, hence $\mathfrak M(B) \in \mathfrak M(\mathcal T(F_t^2)) \cap \band_0$ by Corollary \ref{Cor_band0}. Recall that $\mathfrak M(T_{h_k}^t) \in \band_k$ as seen in Example \ref{exampleA}. In particular, $\mathfrak M(T_{h_k}^t) \in \mathfrak M(\mathcal T(F_t^2)) \cap \band_k$. Therefore, $\mathfrak M(T_{h_k}^t B) \in \mathfrak M(\mathcal T(F_t^2)) \cap \band_k$. Recall that $T_{h_k}^t$ acts on the standard basis as
\begin{align*}
T_{h_k}^t e_j^t = c_j^k e_{j+k}^t,
\end{align*}
where
\begin{align*}
c_j^k = \frac{\Gamma(j + \frac{k}{2} + 1)}{\sqrt{j! (j+k)!}}.
\end{align*}
Using Stirling's approximation
\begin{align*}
\Gamma(x) = \sqrt{\frac{2\pi}{x}} \left( \frac{x}{e} \right)^x (1 + \mathcal O(x)) \quad \text{ as } x \to \infty
\end{align*}
one easily sees that for every $k \geq 1$ we have
\begin{align*}
\lim_{j \to \infty} c_j^k = 1.
\end{align*}
In particular, we obtain that $T_{h_k}^t$ is a compact perturbation of $S^k$, where $S$ is the unilateral shift (in the standard basis). Therefore, $A$ acts as
\begin{align*}
Ae_j^t &= b_j e_{j+k}^t = S^k B e_j^t\\
&= T_{h_k}^tBe_j^t + KBe_j^t
\end{align*}
for $K = S^k - T_{h_k}^t$. $\mathfrak M(K)$ is clearly contained in $\band_k$, and since it is compact we also have $K \in \mathcal T(F_t^2)$ by Proposition \ref{Prop:containment_K}. Therefore, we obtain $A \in \mathcal T(F_t^2)$.

If we assume that $A \in \mathcal T(F_t^2)$ is such that $\mathfrak M(A) \in \band_k$, then we can use essentially the same argument: Upon composing with $T_{\overline{h_{-k}}}^t$, we obtain an operator in $\band_0$ acting, up to a compact perturbation, diagonally with the same coefficients on the standard basis. Then, an application of Corollary \ref{Cor_band0} shows that $c_{j-k}^k \langle Ae_{j+k}^t, e_j^t\rangle_t$ is uniformly continuous with respect to $\rho$. Since $c_{j-k}^k$ converges to $1$ as $j \to \infty$, $(\langle Ae_{j+k}^t, e_j^t\rangle_t)_{j=0}^\infty$ differs by an element in $c_0(\mathbb N_0)$ (the sequences converging to zero) from a uniformly continuous sequence. But every sequence in $c_0(\mathbb N_0)$ is uniformly continuous with respect to $\rho$ (which is easy to verify), hence the statement follows.
\end{proof}
The following result, which is now an easy consequence of the previous proposition, is our second main result and entirely characterizes membership of an operator in $\mathcal T(F_t^2) \cap C_R(F_t^2)$ in terms of its representation with respect to the standard basis of $F_t^2$. Note that, for convenience, we set $e_m^t = 0$ for $m < 0$.
\begin{thm}\label{thm2}
Let $A \in \mathcal L(F_t^2)$. Then, the following are equivalent:
\begin{enumerate}
\item $A \in \mathcal T(F_t^2) \cap C_R(F_t^2)$;
\item $\mathfrak M(A) \in \bdo(\ell^2(\mathbb N_0))$ and $(\langle Ae_j^t, e_{j+k}^t\rangle_t)_{j=0}^\infty$ is uniformly continuous with respect to $\rho$ for every $k \in \mathbb Z$.
\end{enumerate}
\end{thm}
\begin{proof}
This is an easy consequence of the previous proposition and Lemma \ref{lemma_approx_id}.
\end{proof}

\section{Generalizations of the result}\label{sec:3}
\subsection{Separately-rotationally-continuous operators on multi-variable Fock spaces}
Upon considering the Fock space $F_t^2 = F_t^2(\mathbb C^d)$, i.e., the closed subspace of $L^2(\mathbb C^d, \mu_t)$ of holomorphic functions (where now $d\mu_t(z) = \frac{1}{(\pi t)^d}e^{-\frac{|z|^2}{t}}~dz$), the same questions can be asked. Investigating radial operators on multivariate Fock spaces has been (successfully) done in the literature, see for example \cite{Dewage_Olafsson2022} and references therein. There is no problem (up to some more involved multi-index notation) in extending the methods presented above to operators which are rotationally-continuous with respect to separately-radial rotations. By this, we mean the following: For $(\zeta_1, \dots, \zeta_d) \in \mathbb T^d$ we define:
\begin{align*}
    U_{(\zeta_1, \dots, \zeta_d)} g(z_1, \dots, z_d) = g(\zeta_1 z_1, \dots, \zeta_d z_d), \quad g \in F_t^2(\mathbb C^d), ~z_1, \dots, z_d \in \mathbb C.
\end{align*}
Now, we can consider the $C^\ast$-algebra:
\begin{align*}
    C_{SR}(F_t^2(\mathbb C^d)) := \{ A \in \mathcal L(F_t^2(\mathbb C^d)):~\| U_{(\zeta_1, \dots, \zeta_d)} A& U_{(\zeta_1, \dots, \zeta_d)}^\ast - A\| \to 0,\\
    &\zeta_1, \dots, \zeta_d \to 1\}.
\end{align*}
The harmonic analysis of $\mathbb T$ that we employed before can, without introducing any new ideas, be replaced by the harmonic analysis on $\mathbb T^d$. For example, the Fej\'{e}r kernel on $\mathbb T$ is replaced by the $d$-fold tensor product of Fej\'{e}r kernels, acting then on $\mathbb T^d$. The standard basis of $F_t^2(\mathbb C^d)$ is given (using multi-index notation) by the elements:
\begin{align*}
    e_\alpha^t(z) = \sqrt{\frac{1}{t^{|\alpha|} \alpha!}} z^\alpha, ~z \in \mathbb C^d, ~\alpha \in \mathbb N_0^d.
\end{align*}
Considering matrix representations of the operators with respect to this standard basis, the matrix representations are operators on $\ell^2(\mathbb N_0^d)$. Hence, one proves identically:
\begin{thm}
    Let $A \in \mathcal L(F_t^2(\mathbb C^d))$. Then, $A \in C_{SR}(F_t^2(\mathbb C^d))$ if and only if $\mathfrak M(A) \in \bdo(\ell^2(\mathbb N_0^d))$.
\end{thm}
Further, one can also define the Toeplitz algebra over $F_t^2(\mathbb C^d)$, which we denote by $\mathcal T(F_t^2(\mathbb C^d))$. Then, replacing the theorem of Esmeral and Maximenko from \cite{Esmeral_Maximenko2016} by the results in \cite{Dewage_Olafsson2022} for the case of separately-radial symbols, which relates membership in the separately-radial Toeplitz algebra with uniform continuity with respect to the square-root metric on $\mathbb N_0^d$:
\begin{align*}
    \rho_d((m_1, \dots, m_d), (n_1, \dots, n_d)) = |\sqrt{m_1} - \sqrt{n_1}| + \dots + |\sqrt{m_d} - \sqrt{n_d}|.
\end{align*}
This then leads to the following:
\begin{thm}\label{thm:main_result}
    Let $A \in \mathcal L(F_t^2(\mathbb C^d))$. Then, the following are equivalent:
\begin{enumerate}[1)]
    \item $A \in \mathcal T(F_t^2(\mathbb C^d)) \cap C_{SR}(F_t^2(\mathbb C^d))$.
    \item $\mathfrak M(A) \in \bdo(\ell^2(\mathbb N_0^d))$ and $(\langle Ae_\alpha^t, e_{\alpha + \beta}^t\rangle_t)_{\alpha \in \mathbb N_0^d}$ is uniformly continuous with respect to $\rho_d$ for every $\beta \in \mathbb Z^d$.
\end{enumerate}
\end{thm}

In several variables, there is a different version of this problem which we formulate now. Indeed, considering separately radial rotations is quite restrictive. Instead, one could consider the operators
\begin{align*}
    U_V g(z) = g(Vz), \quad g \in F_t^2(\mathbb C^d), ~z \in \mathbb C^d, ~V \in U(n).
\end{align*}
Here, $U(n)$ denotes the group of unitary $n\times n$-matrices. We now define
\begin{align*}
    C_R(F_t^2(\mathbb C^d)) := \{ A \in \mathcal L(F_t^2(\mathbb C^d)): ~U_V A U_V^\ast \to A, ~V \to I\}.
\end{align*}
Clearly, $C_R(F_t^2(\mathbb C^d)) \subset C_{SR}(F_t^2(\mathbb C^d))$. Hence, there is hope to read off from properties of $\mathfrak M(A)$ if $A \in C_R(F_t^2(\mathbb C^d))$. Nevertheless, the tools we used in the present work clearly hinge on the commutativity of the underlying groups, hence are not suitable to understand properties of this action of $U(n)$. Hence, we pose the following open problem:
\begin{problem}\label{prop:1}
    Characterize membership of $A \in \mathcal L(F_t^2(\mathbb C^d))$ in $C_R(F_t^2(\mathbb C^d))$ in terms of properties of its matrix $\mathfrak M(A)$. Characterize membership of $A \in \mathcal L(F_t^2(\mathbb C^d))$ in $C_R(F_t^2(\mathbb C^d)) \cap \mathcal T(F_t^2(\mathbb C^d))$ in terms of properties of its matrix $\mathfrak M(A)$. 
\end{problem}
Variants of this problem can clearly also be asked for the action of the $k$-quasi-radial subgroup of $U(n)$, see \cite{Dewage_Olafsson2022} for details on this group.

\subsection{The Bergman space on the disc}
We denote by $\mathbb D \subset \mathbb C$ the unit disk in the complex plane. For $\lambda > -1$ we consider the measure
\begin{align*}
    dv_\lambda(z) = c_\lambda (1-|z|^2)^\lambda ~dz
\end{align*}
on $\mathbb D$, where $c_\lambda > 0$ is such that $v_\lambda$ turns into a probability measure. The Bergman space $A_\lambda^2(\mathbb D)$ consists of the subspace of $L^2(\mathbb D, v_\lambda)$ consisting of holomorphic functions. As a closed subspace of $L^2(\mathbb D, v_\lambda)$, $A_\lambda^2(\mathbb D)$ comes endowed with the inner product $\langle \cdot, \cdot\rangle_\lambda$. We refer to the two books by Zhu on Bergman spaces for details \cite{Zhu2005, Zhu2007}. Similarly to the situation of the Fock space, one can consider the operators
\begin{align*}
    \mathcal U_\zeta g(z) = g(\zeta z), \quad g \in A_\lambda^2(\mathbb D), ~\zeta \in \mathbb T, ~z \in \mathbb D.
\end{align*}
They form a group of unitary operators with $\mathcal U_\zeta^\ast = \mathcal U_{\zeta}^{-1} = \mathcal U_{\overline{\zeta}}$. One can then investigate the class of operators
\begin{align*}
    C_R(A_\lambda^2(\mathbb D)) := \{ A \in \mathcal L(A_\lambda^2(\mathbb D)): ~\| \mathcal U_\zeta A \mathcal U_\zeta^\ast - A\|_{op} \to 0, ~\zeta \to 1\}.
\end{align*}
Working now with the standard orthonormal basis of $\mathcal A_\lambda^2(\mathbb D)$, given by $f_j^\lambda(z) = \sqrt{\frac{\Gamma(n + \lambda + 2)}{n!\Gamma(\lambda + 2)}} z^n$, one can again identify operators on $A_\lambda^2(\mathbb D)$ with their matrix representation on $\ell^2(\mathbb N_0)$. Exactly as in the one-variable Fock-space case (making suitable but straightforward modifications of the notations involved), one obtains:
\begin{thm}
    Let $A \in \mathcal L(A_\lambda^2(\mathbb D))$. Then, $A \in C_R(A_\lambda^2(\mathbb D))$ if and only if $\mathfrak M(A) \in \bdo(\ell^2(\mathbb N_0))$.
\end{thm}

Denote by $P_\lambda$ the orthogonal projection from $L^2(\mathbb D, v_\lambda)$ to $A_\lambda^2(\mathbb D)$. Then, for $h \in L^\infty(\mathbb D)$, the Toeplitz operator $T_h^\lambda$ is defined on $A_\lambda^2(\mathbb D)$ as $T_h^\lambda(g) = P_\lambda(hg)$. By $\mathcal T(A_\lambda^2(\mathbb D))$ we denote the $C^\ast$-algebra generated by all Toeplitz operators on $A_\lambda^2(\mathbb D)$ with symbols in $L^\infty(\mathbb D)$. 

The main results of \cite{Bauer_etal_2014} (see also \cite{Grudsky_etal_2013}) characterizes the radial operators within the Toeplitz algebra.
\begin{thm}[\cite{Bauer_etal_2014, Grudsky_etal_2013}]
    Let $A \in \mathcal L(A_\lambda^2(\mathbb D))$ be radial and set $b_j = \langle A f_j^\lambda, f_j^\lambda\rangle_\lambda$ for $j \in \mathbb N_0$. Then, $A \in \mathcal T(A_\lambda^2(\mathbb D))$ if and only if $(b_j)_{j=0}^\infty$ is uniformly continuous with respect to the logarithmic metric
    \begin{align*}
        d(m, n) = |\ln(m+1) - \ln(n+1)|.
    \end{align*}
\end{thm}

Using this result, and replacing the use of the Toeplitz operators $T_{h_k}^t$ (respectively $T_{\overline{h_k}}^t$) by $T_{g_k}^\lambda$ with $g_k(z) = z^k$ (respectively $T_{\overline{g_k}}^\lambda$), one can then prove the following result analogous to Theorem \ref{thm:main_result}. We leave the details of this to the interested reader. Here, we let again $f_{m}^\lambda = 0$ for $m < 0$.
\begin{thm}
    Let $A \in \mathcal L(A_\lambda^2(\mathbb D))$. Then, the following are equivalent:
    \begin{enumerate}[1)]
        \item $A \in \mathcal T(A_\lambda^2(\mathbb D)) \cap C_R(A_\lambda^2(\mathbb D))$.
        \item $\mathfrak M(A) \in \bdo(\ell^2(\mathbb N_0))$ and $(\langle A f_j^\lambda, f_{j+k}^\lambda\rangle_\lambda)_{j=0}^\infty$ is uniformly continuous with respect to the logarithmic metric for every $k \in \mathbb Z$. 
    \end{enumerate}
\end{thm}

We want to mention that a result for separately-rotationally-continuous operators probably holds true on appropriate Bergman spaces of several variables. On the Bergman spaces over $\mathbb D^d = \mathbb D \times \dots \mathbb D$ this seems very plausible and even for the Bergman space over the unit ball $\mathbb B^d$ of $\mathbb C^d$ this seems rather likely. Nevertheless, we defer from formulating results for this setting here, as the characterization of separately-radial operators in the appropriate Toeplitz algebras seemingly has not been worked out in the literature yet. An open problem for the Bergman space of the ball $\mathbb B^d$ analogously to Problem \ref{prop:1} can be formulated in the same manner.

\subsection*{Acknowledgements}
We appreciate the reviewer's many comments on the paper, which helped to improve the quality of this work.

\bibliographystyle{amsplain}
\bibliography{References}

@BOOK{Zhu,
	title = {Analysis on {F}ock {S}paces},
	publisher = {Springer US},
	author = {Zhu, K.},
	year = {2012},
	series = {Graduate Texts in Mathematics},
	volume = {263},
	address = {New York},
}

@ARTICLE{Bauer_Isralowitz2012,
	author = {Bauer, W. and Isralowitz, J.},
	title = {{C}ompactness characterization of operators in the {T}oeplitz algebra of the {F}ock space {$F_\alpha^p$}},
	journal = {J. Funct. Anal.},
	year = {2012},
	volume = {263},
	issue = {5},
	pages = {1323-1355},
}

@ARTICLE{Xia,
	author = {Xia, J.},
	title = {{Localization and the Toeplitz algebra on the Bergman space}},
	journal = {J. Funct. Anal.},
	year = {2015},
	volume = {269},
	pages = {781-814},
	issue = {3},
}

@ARTICLE{Berger_Coburn1994,
	author = {Berger, C.~L. and Coburn, L.~A.},
	title = {{Heat Flow and Berezin-Toeplitz Estimates}},
	journal = {Amer. J. Math.},
	year = {1994},
	volume = {116},
	issue = {3},
	pages = {563-590},
}

@ARTICLE{Berger_Coburn1987,
	author = {Berger, C.~L. and Coburn, L.~A.},
	title = {{Toeplitz operators on the Segal-Bargmann space}},
	journal = {Trans. Amer. Math. Soc.},
	year = {1987},
	volume = {301},
	issue = {2},
	pages = {813-829},
}

@BOOK{Zhu2007,
	author = {Zhu, K.},
	title = {{Operator Theory in Function Spaces}},
	publisher = {American Mathematical Society},
	year = {2007},
	edition = {2},
	series = {Mathematical Surveys and Monographs},
	volume = {138},	
}

@ARTICLE{Esmeral_Vasilevski2016,
	author = {Esmeral, K. and Vasilevski, N.},
	title = {{C$^\ast$-algebra generated by horizontal Toeplitz operators on the Fock space}},
	year = {2016},
	journal = {Bol. Soc. Mat. Mex.},
	volume = {22},
	issue = {2},
	pages = {567-582},
}

@BOOK{Diestel_Uhl1977,
	author = {Diestel, J. and Uhl, J.~J.},
	title = {{Vector Measures}},
	year = {1977},
	publisher = {American Mathematical Society},
	series = {Mathematical surveys},
	volume = {15},
}

@incollection{Stroethoff1997,
 author = {Stroethoff, K.},
 title = {The {Berezin} transform and operators on spaces of analytic functions},
 booktitle = {Linear operators. Proceedings of the semester organized at the Stefan Banach International Mathematical Center, Warsaw, Poland, February 7--May 15, 1994},
 pages = {361--380},
 year = {1997},
 publisher = {Warsaw: Polish Academy of Sciences, Inst. of Mathematics},
}

@ARTICLE{Fulsche2020,
	author = {Fulsche, R.},
	title = {{Correspondence theory on $p$-Fock spaces with applications to Toeplitz algebras}},
	journal = {J. Funct. Anal.},
	volume = {279},
	issue = {7},
	pages = {108661},
	year = {2020},
}

@ARTICLE{Esmeral_Maximenko2016,
	author = {Esmeral, K. and Maximenko, E.~A.},
	title = {{Radial Toeplitz Operators on the Fock Space and Square-Root-Slowly Oscillating Sequences}},
	year = {2016},
	journal = {Complex Anal. Oper. Theory},
	volume = {10},
	pages = {1655-1677},
}

@article{Dewage_Olafsson2022,
 author = {Dewage, V. and {\'O}lafsson, G.},
 title = {Toeplitz operators on the {Fock} space with quasi-radial symbols},
 journal = {Complex Anal. Oper. Theory},
 issn = {1661-8254},
 volume = {16},
 number = {4},
 pages = {32},
 note = {Id/No 61},
 year = {2022},
}

@book{Lindner2006,
    author = {Lindner, M.},
    title = {{Infinite Matrices and Their Finite Sections}},
    year = 2006,
    publisher = {Birkh\"auser Verlag},
    address = {Basel, Boston, Berlin},
}

@book{Rabinovich_Roch_Silbermann2004,
	author = {Rabinoch, V.~S. and Roch, S. and Silbermann, B.},
	title = {{Limit Operators and Their Applications in Operator Theory}},
	series = {Oper. Theory Adv. Appl.},
	volume = {150},
	publisher = {Birkh\"auser},
	year = {2004},
}

@article{Bauer_etal_2014,
 author = {Bauer, W. and Herrera Ya{\~n}ez, C. and Vasilevski, N.},
 title = {Eigenvalue characterization of radial operators on weighted {Bergman} spaces over the unit ball},
 fjournal = {Integral Equations and Operator Theory},
 journal = {Integral Equations Oper. Theory},
 issn = {0378-620X},
 volume = {78},
 number = {2},
 pages = {271--300},
 year = {2014},
}

@article{Grudsky_etal_2013,
 author = {Grudsky, S.~M. and Maximenko, E.~A. and Vasilevski, N.~L.},
 title = {Radial {Toeplitz} operators on the unit ball and slowly oscillating sequences},
 fjournal = {Communications in Mathematical Analysis},
 journal = {Commun. Math. Anal.},
 issn = {1938-9787},
 volume = {14},
 number = {2},
 pages = {77--94},
 year = {2013},
}

@article{Berger_Coburn1986,
 author = {Berger, C. A. and Coburn, L. A.},
 title = {Toeplitz operators and quantum mechanics},
 fjournal = {Journal of Functional Analysis},
 journal = {J. Funct. Anal.},
 issn = {0022-1236},
 volume = {68},
 pages = {273--299},
 year = {1986},
}

@misc{Fulsche_Hagger2025,
 author = {Fulsche, R. and Hagger, R.},
 title = {Band-dominated and {Fourier}-band-dominated operators on locally compact abelian groups},
 year = {2025},
 howpublished = {Preprint, {arXiv}:2504.17442 [math.{FA}] (2025)},
}

@book{Stein_Shakarchi,
 author = {Stein, E.~M. and Shakarchi, R.},
 title = {Fourier analysis. {An} {Introduction}},
 fseries = {Princeton Lectures in Analysis},
 series = {Princeton Lect. Anal.},
 volume = {1},
 isbn = {0-691-11384-X},
 year = {2003},
 publisher = {Princeton, NJ: Princeton University Press},
}

@book{Zhu2005,
 author = {Zhu, K.},
 title = {Spaces of holomorphic functions in the unit ball},
 fseries = {Graduate Texts in Mathematics},
 series = {Grad. Texts Math.},
 issn = {0072-5285},
 volume = {226},
 isbn = {0-387-22036-4},
 year = {2005},
 publisher = {New York, NY: Springer},
 language = {English}
}

@article{Dewage_Mitkowski2025,
 author = {Dewage, V. and Mitkovski, M.},
 title = {Density of {Toeplitz} operators in rotation-invariant {Toeplitz} algebras},
 fjournal = {The Journal of Fourier Analysis and Applications},
 journal = {J. Fourier Anal. Appl.},
 issn = {1069-5869},
 volume = {31},
 number = {4},
 pages = {24},
 note = {Id/No 57},
 year = {2025},
}

@article{Kozak_Simonenko1980,
 author = {Kozak, A.~V. and Simonenko, I.~B.},
 title = {Projection methods for the solution of multidimensional discrete equations in convolutions},
 fjournal = {Siberian Mathematical Journal},
 journal = {Sib. Math. J.},
 issn = {0037-4466},
 volume = {21},
 pages = {235--242},
 year = {1980},
}

@article{Cuckovic_Rao1998,
 author = {{\v{C}}u{\v{c}}kovi{\'c}, {\v{Z}}. and Rao, N. V.},
 title = {Mellin transform, monomial symbols, and commuting {Toeplitz} operators},
 fjournal = {Journal of Functional Analysis},
 journal = {J. Funct. Anal.},
 issn = {0022-1236},
 volume = {154},
 number = {1},
 pages = {195--214},
 year = {1998},
}

\end{document}